\documentclass{amsart}[12pt]
\usepackage{mathrsfs, amsmath,amssymb}
\usepackage{mathtools}
\usepackage{fancyhdr}
\newtheorem{teo}{Theorem}
\newtheorem{pro}{Proposition}
\newtheorem{lem}{Lemma}
\newtheorem{cor}{Corollary}
\newtheorem*{rem}{Remark}

\title{Relating moments of self-adjoint polynomials in two orthogonal projections }

\author{Nizar Demni}
\address{ Aix-Marseille Universit\'e CNRS Centrale
Marseille I2M - UMR 7373. 39 rue F. Joliot Curie, 13453 Marseille, France }
\email{nizar.demni@univ-amu.fr}

\author[T. Hamdi]{Tarek Hamdi}
\address{Department of Management Information Systems \\ College of Business Management \\ Qassim University \\ Ar Rass \\ Saudi Arabia
and Laboratoire d'Analyse Math\'ematiques et applications LR11ES11 \\ Universit\'e de Tunis El-Manar \\ Tunisie}
\email{ t.hamdi@qu.edu.sa } 


\begin{document}
\maketitle

\begin{abstract} 
Given two orthogonal projections $\{P,Q\}$ in a non commutative tracial probability space, we prove relations between the moments of $P+Q$, of $\sqrt{-1}(PQ-QP)$ and of $P+QPQ$ and those of the angle operator $PQP$. Our proofs are purely algebraic and enumerative and does not assume $P,Q$ satisfying Voiculescu's freeness property or being in general position. As far as the sum and the commutator are concerned, the obtained relations follow from binomial-type formulas satisfied by the orthogonal symmetries associated to $P$ and $Q$ together with the trace property. In this respect, they extend those corresponding to the cases where one of the two projections is rotated by a free Haar unitary operator or more generally by a free unitary Brownian motion. As to the operator $P+QPQ$, we derive autonomous recurrence relations for the coefficients (double sequence) of the expansion of its moments as linear combinations of those of $PQP$ and determine explicitly few of them. These relations are obtained after a careful analysis of the structure of words in the alphabet $\{P, QPQ\}$. We close the paper by exploring the connection of our previous results to the so-called Kato's dual pair. Doing so leads to new identities satisfied by their moments. 
\end{abstract}


\section{Introduction}
Let $(\mathcal{A}, \tau)$ be a non commutative probability space and assume $\tau$ is a trace. Consider two orthogonal (self-adjoint) projections $P, Q \in \mathcal{A}$ and assume they are free in Voiculescu's sense. Then the spectral distributions of their sum and of their self-adjoint product $PQP$ (angle operator) are the free additive and the free multiplicative convolutions of Bernoulli distributions (\cite{Nic-Spe}). As to the spectral distribution of their self-adjoint commutator 
\begin{equation*}
C:= i(PQ-QP),
\end{equation*}
it was computed in \cite{NS} and related there to that of $PQP$ relying on combinatorics of free cumulants. 

If $P$ and $Q$ are not necessarily free, one may obtain two free projections by rotating for instance $Q$ by a Haar unitary operator $U$, that is by considering $P$ and $UQU^{\star}$. A concrete realisation of this property amounts to consider orthogonal projections with convergent ranks in a finite dimensional complex vector space, to rotate one of them by a Haar unitary matrix and to let the matrix size tend to infinity. Doing so has the merit to allow for the use of Jordan (or principal) angles, as shown for instance in \cite{Aub} (see also \cite{AHT}, \cite{Oml}, \cite{Gal}, \cite{Bak-Tre0}, \cite{Bak-Tre1} and references therein). In particular, the spectra of the sum and of the self-adjoint commutator of the underlying matrix orthogonal projections are related to the spectrum of their angle operator through elementary functions. Then, using the asymptotic freeness property, one carries these relations to the densities of the spectral distributions of the corresponding self-adjoint operators. As to the operator $P + QPQ$, the situation becomes quite difficult even under the freeness assumption since $P$ and $QPQ$ are no longer free. This level of difficulty is for instance transparent from the complicated expression density of its spectral distribution and its support is disconnected. 

Relaxing the freeness assumption, one may ask whether the relations alluded to above remain valid when $P$ and $Q$ are in general position. Recall (see e.g \cite{Hal}) that this property means that the four intersections of the closed subspaces corresponding to $P$ and $Q$ and of their orthogonal complements are trivial. In this case, Theorem 2 in \cite{Hal} provides a unitary representation of $\{P,Q\}$ in the space of two-by-two matrices with coefficients in some Hilbert space of bounded operators. Endowing this matrix algebra with the state $(1/2)\tau \otimes \textrm{Tr}$, we can study spectral distributions of self-adjoint polynomials in $\{P,Q\}$. For instance, the moment sequences of $PQP$ and of $(P+Q-{\bf 1})^2$ coincide up to a factor $2$ (see \cite{Hal}, p. 386), where ${\bf 1}$ is the unit of $\mathcal{A}$. Moreover, if $\tau(P) = \tau(Q) = 1/2$ then this equality between moment sequences says that the spectral distribution of $(P+Q- {\bf 1})^2$ in $(\mathcal{A}, \tau)$ is the same as the one of $PQP$ in the compressed space $(P\mathcal{A}P, 2\tau)$. In particular, the general position property holds true for the orthogonal projections $P$ and $U_tQU_t^{\star}, t > 0,$ with $\tau(P)= \tau(Q) = 1/2$, where $(U_t)_{t \geq 0}$ is a free unitary Brownian motion in $\mathcal{A}$ (\cite{Bia}) which is free from $\{P,Q\}$ (\cite{IU}, Remark 3.5). In this case, the angle process is referred to as the liberation process in $(\mathcal{A}, \tau)$ or as the free Jacobi process  $\{P\mathcal{A}P, 2\tau)$. Besides, the spectral distribution of the latter at any time $t > 0$ may be described through the real part of $U_{2t}$ up to an affine transformation (\cite{DHH}, \cite{IU}, remark 3.4). In a nutshell, one obtains an equality between the spectral distributions of the free Jacobi process in the compressed space and of its additive version in $(\mathcal{A}, \tau)$. Using the terminology of \cite{BL}, one obtains a relation between the multiplicative and the additive $t$-free convolution of $P$ and $Q$ (though we do not require these operators to be classically independent as in \cite{BL}).


In this paper, we relate the moments of $P+Q$, of $C$ and of $P+QPQ$ to those of $PQP$ for arbitrary orthogonal projections, without any freeness or general position assumption. Though our proofs depend on the trace property satisfied by $\tau$, they are purely algebraic and enumerative. Actually, as far as $P+Q$ and $C$ are concerned, they rely on binomial-type formulas satisfied by the orthogonal symmetries associated to $\{P, Q\}$, namely:
\begin{equation*}
R := 2P-{\bf 1}, \quad S := 2Q-{\bf 1}.
\end{equation*}
Doing so carries our problem into relating the moments of $R+S$ and of $i(RS-SR)$ to those of $RS$ which results in simpler computations since $R$ and $S$ are involutions. When applied to free orthogonal projections or to $P$ and $U_tQU_t^{\star}$ with 
$\tau(R) = \tau(S) = 0$, our obtained relations reduce to known results such as Theorem 3.10 in \cite{BL} and Corollary 2 in \cite{DHH}, while they allow for computing the spectral distribution of 
\begin{equation*}
i(PU_tQU_t^{\star} - U_tQU_t^{\star}P),
\end{equation*} 
extending a result due to Nica and Speicher (\cite{NS}). As to relating the moments of $P+QPQ$ to those of $PQP$, this problem turns out to be much more trickier than the two previous ones. This is basically due to the fact one is led to enumerate words in 
$\{P,Q\}$ using the alphabet $\{P, QPQ\}$ subject only to $P^2 = P$. Nonetheless, we succeed to derive autonomous recurrence relations for the double sequence $f(n,k), 0 \leq k \leq n,$ encoding the expansion: 
\begin{equation*}
\tau[(P+QPQ)^n] = \tau(P) + \sum_{k=2}^n f(n,k) \tau(PQ)^k, \quad n \geq 2.
\end{equation*} 
In particular, we shall determine explicitly the sequences $f(n,2), n \geq 2,$ and $f(n,3)$, $n \geq 3,$ and show also that $f(n,n)$ is a Lucas sequence (a Fibonacci sequence with different initial values). However, we did not succeed to find a single general expression of $f(n,k)$ valid for any pair $(n,k)$ and we do not believe it exists regarding the complicated form of the obtained recurrence relations (or equivalently the form of the generating function). Nonetheless, one realises from our computations below how subtle could be the spectral study of a `simple' self-adjoint polynomial in $\{P,Q\}$. Even more, the following problem raises: is there any class of (self-adjoint) polynomials for which it would be `possible to write down explicitly' the relations between their moments and those of the angle operator. 

Though we do not have any insight into this problem, we would like to stress that our approach relying on the couple $(R,S)$ of symmetries is closely related to the so-called Kato's dual pair associated with $(P,Q)$. This pair of self-adjoint operators plays a key role in the analysis of the quantum Hall effect and in perturbation theory (see \cite{ASS0}, \cite{ASS} and references therein), and their squares sum to the unit operator. They also satisfy the following additional remarkable properties: they anti-commute and their squares lie in the center of algebra generated by $(P,Q)$. In particular, it is readily seen that $(P-Q)^2, ({\bf 1} - P - Q)^2$ and $C$ are elementary symmetric polynomials of degree two in the underlying Kato's pair. In particular, our previous results written through the Kato's dual pair show that their odd moments are constant, a fact that reminds Theorem 4.1 in \cite{ASS} on the index of a pair of orthogonal projections. 

The paper is organised as follows. In Section 2, we is prove the moment relation between $(P+Q-1)^2$ and $PQP$ and show there how it reduces to the description of the free Jacobi process proved in \cite{DHH} and valid for $\tau(P) = \tau(Q) = 1/2$. In section 3, we prove that the moments of the square of the commutator $C$ coincide with those of $PQ(P-PQ)$ up to a multiplicative factor $2$ and apply this result to the pair $\{P, U_tQU_t^{\star}\}$. The fourth section is devoted to the analysis of the moment structure of 
$P + QPQ$. In particular, we derive there the recurrence relations satisfied by the double sequence $f(n,k), 0 \leq k \leq n,$ and determine explicitly $f(n,2), n \geq 2, f(n,3), n \geq 3,$ and $f(n,n), n \geq 3,$. In the last section, we recall the Kato's dual pair of 
$(P,Q)$ and rewrite our previously obtained results in order to obtain new identities satisfied by their moments.  

\section{Relating the moments of $PQP$ to those of $(P+Q - {\bf 1})^2$} 
As stated in the previous section, the relation between the moments of $PQP$ and of $(P+Q - {\bf 1})^2$ appeals to the following binomial-type formula satisfied by pairs of involutions, in particular by $R$ and $S$. More precisely, 
\begin{pro}\label{binom}
Let $a,b\in \mathcal{A}$ be two involutions: $a^2=b^2={\bf 1}$. Then, for any $n\geq1$,
\begin{equation*}
(a+b)^{2n}=\binom{2n}{n}{\bf 1}+\sum_{k=1}^n\binom{2n}{n-k}((ab)^k+(ba)^k).
\end{equation*}
\end{pro}
\begin{proof}
We proceed by induction. The formula is obviously true for $n=1$. Let $n\geq 2$ and assume the formula holds true up to rank $n$:
\begin{equation*}
	(a+b)^{2l}=\sum_{k=0}^l \mathcal{C}(l,l-k)((ab)^k+(ba)^k), \quad 1\leq l \leq n,
\end{equation*}
where we set:
\begin{equation*}
	\mathcal{C}(l,k):=
	\begin{cases}
	 \displaystyle 	\binom{2l}{k}&,0\le k\le l-1\\
	 \displaystyle 	\frac{1}{2}\binom{2l}{l}&,k=l
	\end{cases}.
\end{equation*}

Then,
\begin{align*}
	(a+b)^{2n+2} =& (2{\bf 1}+ab+ba)\sum_{k=0}^n \mathcal{C}(n,n-k)((ab)^k+(ba)^k)
	\\=&2\mathcal{C}(n,n)(2{\bf 1}+ab+ba)+\sum_{k=1}^n \mathcal{C}(n,n-k)[2(ab)^k+2(ba)^k
	\\&+(ab)^{k+1}+(ba)^{k+1}+(ab)^{k-1}+(ba)^{k-1}]
	\\=&[4\mathcal{C}(n,n)+2\mathcal{C}(n,n-1)]{\bf 1}+[2\mathcal{C}(n,n)+2\mathcal{C}(n,n-1)+\mathcal{C}(n,n-2)](ab+ba)
	\\&+\sum_{k=2}^{n-1}[\mathcal{C}(n,n-k+1)+2\mathcal{C}(n,n-k)+\mathcal{C}(n,n-k-1)]((ab)^k+(ba)^k)
	\\&+[2\mathcal{C}(n,0)+\mathcal{C}(n,1)]((ab)^n+(ba)^n)+\mathcal{C}(n,0)((ab)^{n+1}+(ba)^{n+1}).
\end{align*}
The proposition follows from elementary properties of binomial coefficients.
\end{proof}
With the help of this proposition, we are able to relate the moments of the self-adjoint operator $(R+S)$ to those of the unitary operator $RS$: 
\begin{cor}\label{moments1}
For any $n\geq0$,
\begin{align*}
\tau[(R+S)^{n}]&=
	\begin{cases}\displaystyle
		\binom{2j}{j}+2\sum_{k=1}^{j}\binom{2j}{j -k}\tau[(RS)^k], & n = 2j\\
		\displaystyle2^{2j}\tau(R+S), & n = 2j+1
	\end{cases}.
\end{align*}
\end{cor}
\begin{proof}
The even moments follows readily from Proposition \ref{binom} applied to $a=R$ and $b= S$ together with the trace property of $\tau$. As to the odd ones, note that that if $a,b\in \mathcal{A}$ are such that $a^2=b^2={\bf 1}$, then
\begin{equation*}
	\tau(a(ab)^k)=\tau(a(ba)^k)=\begin{cases}
		\tau(a)&,k\ even\\
		\tau(b)&,k\ odd
	\end{cases}
\end{equation*}
and similarly
\begin{equation*}
	\tau(b(ab)^k)=\tau(b(ba)^k)=\begin{cases}
		\tau(a)&,k\ odd\\
		\tau(b)&,k\ even
	\end{cases}.
\end{equation*}
Consequently, Proposition \ref{binom} again yields: 
\begin{align*}
	\tau[(R+S)^{2n+1}]&=\tau[(R+VSV^{\star})(R+VSV^{\star})^{2n}]
	\\&=\tau(R+S)\left[\binom{2n}{n}+2\sum_{k=1}^n\binom{2n}{n-k}\right]
	\\&=\tau(R+S)\sum_{k=-n}^n\binom{2n}{n+k} = 2^{2n}\tau(R+S)
\end{align*}
as claimed.
\end{proof}

According to this corollary, the equality 
\begin{equation*}
\tau[(R+S)^{2j}] = \binom{2j}{j}+2\sum_{k=1}^{j}\binom{2j}{j -k}\tau[(RS)^k],
\end{equation*}
holds for any $j \geq 0$ (an empty sum is zero). On the other hand, the proof of Proposition 4.1 in \cite{Ham1} shows that: 
\begin{equation*}
\tau[(PQP)^j] = \frac{1}{2^{2j+1}}\binom{2j}{j} + \frac{\tau(R+S)}{4} + \frac{1}{2^{2j}}\sum_{k=1}^{j}\binom{2j}{j -k}\tau[(RS)^k)], \quad j \geq 0.
\end{equation*}
Comparing both formulas and noting that $(R+S)/2 = P+Q-{\bf 1}$, we end up with the sough moment relation:
\begin{teo}
For any $j \geq 1$,
\begin{equation}\label{Eq1}
2\tau[(PQP)^j] - \frac{\alpha+\beta}{2} = \frac{1}{2^{2j}}\tau[(R+S)^{2j}] = \tau[(P+Q - {\bf 1})^{2j}]. 
\end{equation}
\end{teo}
In particular, setting
\begin{equation*}
\alpha := \tau(R) = 2\tau(P) - 1, \quad \beta := \tau(S) = 2\tau(Q) - 1, 
\end{equation*}
and if $\tau(P) = \tau(Q) = 1/2$ then $\alpha = \beta = 0$ whence
\begin{equation*}
\frac{1}{\tau(P)}\tau[(PQP)^j] = \tau[(P+Q - {\bf 1})^{2j}]. 
\end{equation*}
Consequently, the spectral distribution of $PQP$ in the compressed space $(P\mathcal{A}P, 2\tau)$ is the push-forward of the spectral distribution of $P+Q$ under the map $x \mapsto (x-1)^2$. 
In particular, this holds whenever $P$ and $Q$ are in general position as follows from Halmos two projections Theorem (\cite{Hal}, see Theorem 2, p.384). 
This is also in agreement with Theorem in \cite{Aub} where the projections are assumed to be free in $\mathcal{A}$ and applies more generally to the so-called free Jacobi process as shown in the next paragraph. 

On the other hand, replacing $P,Q$ by their orthogonal complements ${\bf 1} - P, {\bf 1} - Q$, one readily gets 
\begin{equation}\label{Eq10}
2\tau[(({\bf 1}-P)({\bf 1}-Q)({\bf 1} - P))^j] + \frac{\alpha+\beta}{2} =  \tau[(P+Q - {\bf 1})^{2j}]
\end{equation}
whence the following identity: 
\begin{equation}\label{Identity-Diff}
2\tau[(PQP)^j] - 2\tau[(({\bf 1}-P)({\bf 1}-Q)({\bf 1} - P))^j]  = \alpha+\beta.
\end{equation}
If we only replace $P$ by ${\bf 1}-P$ then \eqref{Eq1} yields 
\begin{equation}\label{Eq1}
2\tau[(({\bf 1}-P)Q{\bf 1}-P))^j] + \frac{\alpha-\beta}{2} = \frac{1}{2^{2j}}\tau[(R-S)^{2j}] = \tau[(P-Q)^{2j}]. 
\end{equation}
In the study of relative positions of finite-dimensional subspaces, $(P+Q-{\bf 1})^2$ and $(P-Q)^2$ are referred to as the closeness and separation operators (\cite{Dav}, \cite{Gal}). More generally, they are squares of the cosine and the sine operators given by Halmos Theorem.

\subsection{The free Jacobi process and its additive counterpart}
The free Jacobi process $(J_t)_{t \geq 0}$ is defined as $(PQ_tP)_{t \geq 0}$ with $Q_t$ being of the form $U_tQ'U_t^{\star}$, where $Q' \in \mathcal{A}$ is an orthogonal projection, $(U_t)_{t \geq 0} \in \mathcal{A}$ is a free unitary Brownian motion (\cite{Bia}) and is assumed to be $\star$-free from $\{P, Q'\}$. Note however that for any fixed $t > 0$, $P$ and $Q_t$ are not free while they are in the limiting regime $t \rightarrow +\infty$ since $U_t$ weakly converges to a Haar unitary operator $U_{\infty} = U$.
Moreover, $J_t$ is a self-adjoint operator valued in the compressed probability space
\begin{equation*}
\left(P\mathcal{A}P, \tau/\tau(P)\right).
\end{equation*}
In this respect, it was proved in \cite{DHH} and \cite{IU} that if $\alpha = \beta = 0$ then $J_t$ is distributed in the compressed space $\left(P\mathcal{A}P, 2\tau\right)$ as 
\begin{equation}\label{Jacobi}
\frac{2{\bf 1} + U_{2t} + U_{2t}^{\star}}{4} 
\end{equation}
in $(\mathcal{A}, \tau)$. From a geometrical perspective, $[\arg(U_{2t})]/2 \in [0,\pi/2]$\footnote{The spectral distribution of $U_t$ in invariant under complex conjugation so that it is completely determined by its restriction to the upper half of the unit circle.} is the infinite dimensional Jordan angle $\Theta_t$ between $P$ and $(Q_t)$ and the identity \eqref{Eq1} shows that $|P + Q_t - {\bf 1}|$ is distributed as $\cos(\Theta_t)$.

The description \eqref{Jacobi} may be proved from the first equality displayed in \eqref{Eq1} since the spectral distributions of $(R+S)^2$ in 
$(\mathcal{A}, \tau)$ and of $4PQP$ in $(P\mathcal{A}P,  2\tau)$ coincide when $\alpha = \beta = 0$. Accordingly, setting $S' = 2Q'-{\bf 1}$ then 
\begin{equation*}
S_t ;= 2Q_t - {\bf 1} = U_tS'U_t^{\star} 
\end{equation*}
and $RS_t = RU_tS'U_t^{\star}$ is distributed as $U_{2t}$ starting at $RS'$. Moreover,
\begin{equation*}
(R+S_t)^2 = 2{\bf 1} + RU_tS'U_t^{\star} + (RU_tS'U_t^{\star})^{\star}.
\end{equation*} 
Note that similar results were proved in \cite{BL} in relation to the so-called $t$-freeness interpolating between the classical independence and Voiculescu's freeness property.  

\section{Relating the moments of $C$ and of $PQP$}
The commutator of two self-adjoint operators plays a key role in both mathematics and mathematical physics. In the free probability realm, the distribution of the commutator of two free variables was determined in \cite{NS} relying on combinatorics of non crossing partitions. In \cite{Aub}, the author appeals to the asymptotic freeness property to retrieve Nica and Speicher's description of the commutator of two free orthogonal projections. Here, we consider 
\begin{equation*}
C= i(PQ - QP),
\end{equation*}
without assuming that $P$ and $Q$ are free and prove the following equality: 
\begin{teo}\label{Th2}
For any $j \geq 1$,
\begin{equation*}
\tau(C^{2j}) = 2\tau[(PQP(P-PQP))^j].
\end{equation*}
\end{teo}
\begin{proof} 
Use $R = 2P-{\bf 1}, S = 2Q-{\bf 1}$ to write
\begin{equation*}
C = \frac{\sqrt{-1}}{4}(RS -  SR).
\end{equation*}
Up to a multiplicative factor, this is the imaginary part of the unitary operator $RS$. Now, note that $RS$ and $(RS)^{-1} = SR$ have the same distribution since $\tau$ is tracial. 
Consequently, $C$ is an even element: its odd moments vanish. Indeed, 
\begin{align*}
\tau(C^n) = \left( \frac{i}{4}\right)^n \sum_{k=0}^n(-1)^{n-k}\binom{n}{k} \tau[(RS)^{2k-n}],
\end{align*}
and if $n$ is odd then 
\begin{align*}
2\sum_{k=0}^n(-1)^{n-k}\binom{n}{k} \tau[(RS)^{2k-n}] & = \sum_{k=0}^n\left\{(-1)^{n-k}\binom{n}{k} \tau[(RS)^{2k-n}] + (-1)^{k}\binom{n}{k} \tau[(RS)^{n-2k}]\right\} 
\\&= \sum_{k=0}^n[(-1)^{n-k} + (-1)^k]\binom{n}{k} \tau[(RS)^{2k-n}] = 0.
\end{align*}
Otherwise, $n = 2p \geq 2$ is even and the even moments of $A$ are given by: 
\begin{align*}
\tau(C^{2j}) & = \frac{1}{16^j} \left\{\binom{2j}{j} + 2\sum_{k=0}^{j-1}(-1)^{j-k}\binom{2j}{k} \tau[(RS)^{2(j-k)}]\right\} 
\\& = \frac{1}{16^j} \left\{\binom{2j}{j} + 2\sum_{k=1}^{j}(-1)^{k}\binom{2j}{j-k} \tau[(RS)^{2k}]\right\}. 
\end{align*}
Next, consider the operator 
\begin{equation*}
({\bf 1}+R)({\bf 1}+S)({\bf 1}+R)({\bf 1}-S).
\end{equation*}
Since $({\bf 1}+S)({\bf 1}-S) = 0$, then this operator reduces to
\begin{equation*}
({\bf 1}+R)({\bf 1}+S)R({\bf 1}-S),
\end{equation*}
and since $\tau$ is a trace, then the moments of the latter coincide with those of
\begin{equation*}
({\bf 1}-S)({\bf 1}+R)({\bf 1}+S)R = ({\bf 1}-S)R({\bf 1}+S)R. 
\end{equation*}
But
\begin{align*}
({\bf 1}-S)R({\bf 1}+S)R &= (R-SR)(R+SR) 
\\& = {\bf 1} +RSR - S - (SR)^2 
\\& = ({\bf 1}+RSR)({\bf 1}-S).
\end{align*}
Consequently, for any $j \geq 1$
\begin{align*}
\tau[[({\bf 1}+R)({\bf 1}+S)({\bf 1}+R)({\bf 1}-S)]^j] & = \tau[[({\bf 1}+RSR)({\bf 1}-S)]^j]
\\& = \frac{1}{2}\binom{2j}{j} + 2^{2n-2} \tau(RSR -S) \\& + \sum_{k=1}^{j}(-1)^{k}\binom{2j}{j-k} \tau[(RS)^{2k}],
\end{align*}
where the last equality follows from the proof of Proposition 4.1 in \cite{Ham1}. Keeping in mind $R = 2P-{\bf 1}, S = 2Q-{\bf 1}$ and using again the trace property of $\tau$, we further get:
\begin{equation*}
\frac{2}{16^j} \tau[[({\bf 1}+R)({\bf 1}+S)({\bf 1}+R)({\bf 1}-S)]^j] = 2 \tau[(PQ(P-PQ))^j]  = \tau(C^{2j}).
\end{equation*}
Finally, noting that $PQ(P-PQ) = PQP(P - PQ)$ and since $\tau$ is tracial, we get the equality:
\begin{equation*}
\tau[(PQP(P-PQ))^j] = \tau[(PQP(P-PQP))^j] 
\end{equation*}
proving the theorem.
\end{proof}
In particular, if $\tau(P) = \tau(Q) = 1/2$ then the distribution of $C^2$ in $(\mathcal{A}, \tau)$ is the pushforward of the distribution of $PQP(P-PQP)$ in the compressed space $(P\mathcal{A}P, 2\tau)$ under the map $x \in x(1-x), x \in [0,1]$. 
This is in agreement with Example 1 from \cite{Aub}. More generally, we get the following by-product: 
\begin{cor}
For any $t > 0$, the square of the commutator 
\begin{equation*}
C_t:=  i(PU_tQU_t^{\star} - U_tQU_t^{\star}P)
\end{equation*}
has the same spectral distribution as: 
\begin{equation*}
\frac{[2{\bf 1} + U_{2t} + U_{2t}^{\star}][2{\bf 1} - (U_{2t} + U_{2t}^{\star})]}{16}.
\end{equation*}
\end{cor}
\section{Yet another polynomial: $P+QPQ$} 
So far, we considered the sum, the angle operator and the self-adjoint commutator of two projections, which are basic examples of self-adjoint polynomials in $(P,Q)$. By the virtue of what we already proved, it is natural to tackle the problem of describing spectral distributions of an arbitrary self-adjoint polynomials. However, the complexity of this problem may increase drastically even for `simple' polynomials such as $P + QPQ$. This polynomial was considered in \cite{Aub} subject to the freeness of $\{P,Q\}$ and one already realises that the density of the corresponding spectral distribution admits a more complicated expression compared to those corresponding the previous self-adjoint polynomials (its support is not connected). 

Apparently, replacing $(P,Q)$ by $(R,S)$ does not make the problem easier. In order to have more insight into the structure of the moments of $P + QPQ$, one proceed as follows. Firstly, an induction shows that for any $n \geq 2$, the expansion $(P+QPQ)^n$ contains at most the following factors: 
\begin{equation}\label{terms}
P, \quad \{(PQ)^k, (QP)^k\}_{k=2}^n, \quad \{P(QP)^k\}_{k=2}^{n-1}, \quad \{Q(PQ)^k\}_{k=2}^n,
\end{equation}
where the third set is empty for $k=2$. 
Indeed, $(P+QPQ)^2 = P + (PQ)^2 + (QP)^2 + Q(PQ)^2$. Moreover, assuming this claim holds true up to order $n \geq 2$, then the induction is readily checked from the arguments below:
\begin{itemize}
\item $P, \{(PQ)^k\}_{k=2}^{n}, \{P(QP)^k\}_{k=2}^{n-1}$ are invariant by multiplication to the left by $P$. 
\item $(PQ)^{n+1} = P [Q(PQ)^n]$ and $P(QP)^n = P[(QP)^n]$.
\item $(QPQ)[(PQ)^k] = Q(PQ)^{k+1}, 2 \leq k \leq n$, while $Q(PQ)^2 = [QPQ][QPQ]$ appears only when $n=2$.
\item $QPQ[(QP)^k] = (QP)^{k+1}, 2 \leq k \leq n,$ while $(QP)^2 = [QPQ]P$.
\end{itemize}
Here are the first few expansions:
\begin{equation*}
(P+QPQ)^2 = P + (PQ)^2 + (QP)^2 + Q(PQ)^2, 
\end{equation*}
\begin{align*}
(P+QPQ)^3 & = P + (PQ)^2 + (QP)^2 + P(QP)^2 \\& 
+(PQ)^3 + (QP)^3 + 2Q(PQ)^3, 
\end{align*}
\begin{align*}
(P+QPQ)^4 & = P + (PQ)^2 + (QP)^2 + 2P(QP)^2 + P(QP)^3 \\& 
+(PQ)^3 + (QP)^3 + Q(PQ)^3 \\& 
+2(PQ)^4 + 2(QP)^4 + 3Q(PQ)^4.
\end{align*}
Secondly, the previously proved claim together with the trace property satisfied by $\tau$ show that the moments of $P + QPQ$ may be written as: 
\begin{equation*}
\tau[(P+QPQ)^n] = \tau(P) + \sum_{k=2}^n f(n,k) \tau(PQ)^k, \quad n \geq 2,
\end{equation*}
where $f(n,k) > 0$. 
Finally and most importantly, we need to compute $f(n,k), 2 \leq k \leq n$. In this respect, we shall prove the following recurrence relations: 

\begin{teo}\label{RecRel}
	The family $(f(n,k))_{2\le k\le n}$ is characterised by the following identities:
	\begin{itemize}
		\item $(f(n,n))_{n \geq 1}$ is the Lucas sequence: for any $n \geq 3$,
		\begin{equation*}\label{k=n}
			f(n,n)=f(n-1,n-1)+f(n-2,n-2), \quad f(1,1) = 1, f(2,2) = 3.
		\end{equation*}
		\item If $k\in\{2,3\}$ then 
		\begin{equation*}
			f(n,k)=n+\delta_{n,k}, \quad n \geq k.
		\end{equation*}
		\item For any $n\ge4$ and $k\in\{3,\ldots,n-1\}$,
		\begin{equation*}\label{general}
			f(n,k)=f(n-1,k-1)+f(n-1,k)-f(n-2,k-1)+f(n-2,k-2),
		\end{equation*}
		where $f(n,1) = 0$ for any $n \geq 4$. 
			\end{itemize}
\end{teo}
The proof of this theorem relies on the two lemmas proved below.   
\begin{lem}\label{a+b}
Denote $a(n,k), b(n,k), c(n,k), d(n,k)$ the cardinalities of 
\begin{equation*}
\{P(QP)^k\}_{k=2}^{n-1}, \quad \{Q(PQ)^k\}_{k=2}^n, \quad \{(PQ)^k\}_{k=2}^n, \quad \{(QP)^k\}_{k=2}^n, \quad ,
\end{equation*}
respectively. Then 
\begin{eqnarray*}
a(n,k) & = & a(n-1,k) + d(n-1,k) \\ 
b(n,k) & = & b(n-1,k-1) + c(n-1,k-1)\\
c(n,k) & = & c(n-1,k) + b(n-1,k-1) \\ 
d(n,k) & = & d(n-1,k-1) + a(n-1,k-2).
\end{eqnarray*}
Consequently,
\begin{equation}\label{RelationF}
a(n+1,k) + b(n+1,k+1) = f(n,k). 
\end{equation}
\end{lem}
\begin{proof}
The recurrence relations follow readily from the identities: 
\begin{eqnarray*}
P(QP)^k &=& P[P(QP)^k]  =  P[(QP)^k] , \\
Q(PQ)^k &=& QPQ[Q(PQ)^{k-1}] = QPQ[(PQ)^{k-1}],\\
(PQ)^k &=& P(PQ)^k = P[Q(PQ)^{k-1}], \\ 
(QP)^k &=& QPQ([QP)^{k-1}] =  QPQ[P(QP)^{k-2}].
\end{eqnarray*}
As to \eqref{RelationF}, it suffices to notice that 
\begin{align*}
f(n,k) &= a(n,k) + b(n,k) + c(n,k) + d(n,k) 
\end{align*}
and that the sum of the two last recurrence relations is $a(n+1,k) + b(n+1,k+1) = a(n,k) + d(n,k) + b(n,k) + c(n,k)$. 
\end{proof}

Now, we state and prove the second needed lemma.
\begin{lem}\label{sys}
	The family $(f(n,k))_{2\le k\le n}$ satisfies:
	\begin{itemize}
	\item If $n \geq 2$ then 
	\begin{equation*}
		f(n,n)=f(n-1,n-1)+a(n-1,n-2)+b(n-1,n-1).
	\end{equation*}
	\item If $k=2, n \geq 3,$ then 
	\begin{equation*}
		f(n,2)=f(n-1,2)-b(n-1,2)+1.
	\end{equation*}
	\item If $k=3, n \geq 4,$ then 
	\begin{equation*}
		f(n,3)=f(n-1,2)+f(n-1,3)-b(n-1,3)-a(n-1,2)+b(n-1,2).
	\end{equation*}
	\item For any $n \geq 5, k\in\{4,\ldots,n-1\}$,
	\begin{equation*}
		f(n,k)=f(n-1,k-1)+f(n-1,k)-b(n-1,k)-a(n-1,k-1)+a(n-1,k-2)+b(n-1,k-1).
	\end{equation*}
	\end{itemize}
\end{lem}
\begin{proof}
	We proceed by induction on $n \geq 2$. The case $n=2$ is readily checked from 
	\begin{equation*}
	f(2,2) = 3, \quad f(1,1) = a(1,0) = b(1,1) = 1. 
	\end{equation*}
	Next, assume that the relations above hold true up to order $n$ and write: 
		\begin{align*}
		\tau[(P+QPQ)^{n+1}]=\tau[P(P+QPQ)^{n}]+\tau[QPQ(P+QPQ)^{n}]. 
	\end{align*}
	Setting
	
		\begin{equation*}
		\tau[P(P+QPQ)^n] = \tau(P) + \sum_{k=2}^n g(n+1,k) \tau(PQ)^k, \quad n \geq 2,
	\end{equation*}
	and
	\begin{equation*}
		\tau[QPQ(P+QPQ)^n] =  \sum_{k=2}^n h(n+1,k) \tau(PQ)^k, \quad n \geq 2,
	\end{equation*}
	then it follows that 
	\begin{equation}\label{Sum0}
		f(n+1,k) = g(n+1,k) + h(n+1,k), \quad k \in \{2,\ldots,n+1\}.
	\end{equation}
	Now it is clear that the contribution of $ (PQ)^k,(QP)^k, P(QP)^k$ remains invariant by multiplication to the left by $P$, while the contribution of $Q(PQ)^k$ becomes $\tau(QP)^{k+1}$. Consequently, we readily get: 
	\begin{equation}\label{Recg}
		g(n+1,k)=\begin{cases}
			f(n,2)-b(n,2),& k=2, \\
			f(n,k)-b(n,k)+b(n,k-1),& 3\le k\le n, \\
			b(n,n)& k=n+1.
		\end{cases}
	\end{equation}
	
	On the other hand, the trace property of $\tau$ yields: 
	\begin{eqnarray*}
		\tau[(QPQ)P^n ]&=& \tau[(PQ)^{2}], \\ 
		\tau[(QPQ)(PQ)^k ]&=& \tau[(PQ)^{k+1}], \\ 
		\tau[(QPQ)(QP)^k ]&=& \tau[(PQ)^{k+1}] , \\
		\tau[(QPQ)Q(PQ)^k ]&=& \tau[(PQ)^{k+1}],\\
		\tau[(QPQ)P(QP)^k ]&=& \tau[(PQ)^{k+2}] 
	\end{eqnarray*}
	whence
	\begin{equation}\label{Rech}
		h(n+1,k)=\begin{cases}
			1, &k=2\\
			f(n,2)-a(n,2), &k=3,\\
			f(n,k-1)-a(n,k-1)+a(n,k-2),& 4\le k\le n, \\
			f(n,n)+a(n,n-1)& k=n+1.
		\end{cases}
	\end{equation}
	Combining \eqref{Sum0}, \eqref{Recg} and \eqref{Rech}, we end up with: 
	\begin{equation*}
		f(n+1,k)=\begin{cases}
			f(n,k)-b(n,k)+1, &k=2\\
			f(n,2)+f(n,3)-b(n,3)-a(n,2)+b(n,2), &k=3\\
			f(n,k-1)+f(n,k)-b(n,k)-a(n,k-1)+a(n,k-2)+b(n,k-1),& 4\le k\le n\\
			f(n,n)+a(n,n-1)+b(n,n)& k=n+1,
		\end{cases}
	\end{equation*}
	as desired. 
	\end{proof}
	
\begin{rem}
The initial values are readily read from the above expansions. For instance,
	\begin{equation*}
	a(2,0) =1, \quad b(2,1)  = 0,   \quad a(2,2) = 0, \quad	b(2,2) =1.
	\end{equation*}
	\end{rem}	
We are now ready to prove Theorem \ref{RecRel}.
\begin{proof}[Proof of Theorem \ref{RecRel}]
	Combining Lemma \ref{a+b} and Lemma \ref{sys}, we readily get
	\begin{align*}
		f(n,n) &= f(n-1,n-1)+a(n-1,n-2)+b(n-1,n-1)
		\\& =f(n-1,n-1)+f(n-2,n-2), \quad n \geq 3. 
	\end{align*}
	Furthermore, Lemma \ref{a+b} entails 
	\begin{equation*}
	a(n-1,2) + b(n-1,3) = f(n-2, 2), 
	\end{equation*}
	and 
	\begin{equation*}
	b(n-1,2) = b(n-1,2) + a(n-1,1) = f(n-2,1).
	\end{equation*}
	since $a(n-1,1) = 0, n \geq 3$. As a result, 
	\begin{equation}\label{k=3}
		f(n,3)=f(n-1,2)+f(n-1,3)-f(n-2,2) + f(n-2,1),  \quad n\ge 3,
	\end{equation}
	and similarly 
	\begin{equation}\label{k=2}
		f(n,2)=f(n-1,2)- f(n-2,1) +1, \quad n\ge3.
	\end{equation}
	But it is easy to see that $f(n-2,1) = \delta_{n3}$ whence we infer 
	\begin{eqnarray*}
	f(3,3) & = & f(2,2) + f(1,1) = 4, \\ 
	f(n,3)&=& f(n-1,3) + f(n-1,2)-f(n-2,2), \quad n \geq 4, \\ 
	f(3,2) & = & f(2,2) = 3, \\
	f(n,2) & =& f(n-1, 2) + 1, \quad n \geq 4.
	\end{eqnarray*} 
	Consequently, $f(n,2) = n + \delta_{n2}, n \geq 2$ which in turn implies 
	\begin{equation*}
	f(n,3)= f(n-1,3) + 1 - \delta_{n4} \quad n \geq 4,
	\end{equation*}
	and leads to the expression of $f(n,3)$. 
	 Finally, if $n \geq 4$ and $k\in\{3,\ldots,n-1\}$ then  Lemma \ref{a+b} and Lemma \ref{sys} give the relation: 
	\begin{equation*}
		f(n,k)=f(n-1,k-1)+f(n-1,k)-f(n-2,k-1)+f(n-2,k-2),
	\end{equation*}
	which ends the proof of the theorem.
\end{proof}
\begin{rem}
Let $n \geq 3$. If we take into account the values 
\begin{equation*}
f(n-2, 0) = 1, n \geq 3, \quad  f(n-2,1) = 0, n \geq 4, \quad f(1,1) = 1, 
\end{equation*}
and the fact that $f(n,k) = 0$ whenever $k > n$ then we get the single recurrence relation: 
\begin{equation*}
		f(n,k)= f(n-1,k) + f(n-1,k-1)-f(n-2,k-1)+f(n-2,k-2),
	\end{equation*}
	for any $2 \leq k \leq n$. 
We can convert it into a generating series: if 
\begin{equation*}
G(z,w) := \sum_{n \geq 3}\sum_{k=2}^n f(n,k) z^nw^k, 
\end{equation*}
in a neighborhood of $(0,0)$, then lengthy but routine computations yield the expression: 
\begin{equation*}
G(z,w)[1-z-zw+z^2w-z^2w^2] = 3z^3w^2[1-z+zw] + \frac{z^3w^2(3-2z)}{1-z}.
\end{equation*}
\end{rem}
\begin{rem}
The Lucas sequence admits the following expression: 
\begin{equation*}
f(n,n) = \left(\frac{1+\sqrt{5}}{2}\right)^n +  \left(\frac{1-\sqrt{5}}{2}\right)^n, \quad n \geq 1.
\end{equation*}
\end{rem}
\begin{rem}
One may further compute 
\begin{equation*}
f(n,4) = \frac{n(n-1)}{2} + \delta_{n4} = \binom{n}{2} + \delta_{n4}, \quad n \geq 4,
\end{equation*}
and think that there is a single expression of $f(n,k)$ for any $2 \leq k \leq n$. However, we believe this is not true since for instance 
\begin{equation*}
f(n,5) = n^2 - 4n + 6
\end{equation*} 
is irreducible over $\mathbb{R}$.
\end{rem}	
\section{Kato's Dual pair} 
Given a pair $(P,Q)$ of two orthogonal projections, its Kato's dual $(A,B)$ is defined by: 
\begin{equation*}
A = P-Q, \quad B = {\bf 1} - (P+Q). 
\end{equation*}
The importance of this pair stems from the following relations: 
\begin{equation*}
A^2 + B^2 = {\bf 1}, \quad AB+BA = 0,
\end{equation*}
and from the fact that $B^2$ (and so $A^2$) commutes with $P$ and $Q$ since 
\begin{align}\label{Square}
B^2 = ({\bf 1}-P)({\bf 1} - Q) + QP = ({\bf 1} - Q)({\bf 1}-P) + PQ.
\end{align}
In particular, $B^2P=PB^2 = PQP$ is the angle operator and similarly $B^2Q = QB^2 = QPQ$. Note also that 
\begin{equation*}
B+A = {\bf 1} - 2Q = -S, \quad B-A = {\bf 1} - 2P = -R,
\end{equation*}
so that polynomials in $(A,B)$ are also polynomials in $(R-S, R+S)$.  
Now, \eqref{Eq1} may be written as 
\begin{equation}\label{Eq11}
\tau[B^{2j}] = 2\tau[(PQP)^j] - \frac{\alpha+\beta}{2} = 2\tau[(B^2P)^j] + \tau(B) 
\end{equation}
Note that \eqref{Square} and induction yield the following expressions:
\begin{eqnarray*}
B^{2j} & = & (QP)^j + (({\bf 1}-P)({\bf 1} - Q))^j, \\ 
B^{2j+1} &= &  (({\bf 1}-P)({\bf 1} - Q))^j({\bf 1}-P) - (QP)^jQ.
\end{eqnarray*}
Together with \eqref{Eq11} and the trace property, they imply \eqref{Identity-Diff} which may be written: 
\begin{equation*}
\tau[ (({\bf 1}-P)({\bf 1} - Q))^j] = \tau[(QP)^j] +  \tau(B), 
\end{equation*}
or equivalently 
\begin{equation*}
\tau[B^{2j+1}] =  \tau(B), \quad j \geq 0.
\end{equation*}
Substituting $P \rightarrow {\bf 1} - P$, we similarly get $\tau[A^{2j+1}] =  \tau(A)$ for any $j \geq 0$. This constancy of the odd moments of $A$ and $B$ reminds Theorem 4.1 in \cite{ASS} on the index of a pair of orthogonal projections.  

As to the commutator $C = \sqrt{-1}(PQ-QP)$, it can be written as 
\begin{equation*}
C = \sqrt{-1}(P-Q)(P+Q-{\bf 1}) = -\sqrt{-1}AB =  \frac{\sqrt{-1}}{4}(R-S)(R+S). 
\end{equation*}
Since $A$ and $B$ anti-commute, then we can show that $\tau(C^{2j+1}) = 0, j \geq 0$ while $C^2 = A^2B^2 = B^2({\bf 1} -B^2)$. Combined with Theorem \ref{Th2} and recalling $B^2P = PQP$, we arrive at the following identity: for any $j \geq 1$,
\begin{equation*}
\tau[B^{2j}({\bf 1}-B^2)^j] = 2\tau[(B^2P({\bf 1}-B^2)P)^j] = 2\tau[B^{2j}({\bf 1}-B^2)^jP].
\end{equation*}
Writing 
\begin{equation*}
\tau[B^{2j}({\bf 1}-B^2)^j] = \tau[B^{2j}({\bf 1}-B^2)^jP] + \tau[B^{2j}({\bf 1}-B^2)^j({\bf 1}-P)], 
\end{equation*}
the last identity is equivalent to the following one:
\begin{equation*}
\tau[B^{2j}({\bf 1}-B^2)^jP] = \tau[B^{2j}({\bf 1}-B^2)^j({\bf 1}-P)]
\end{equation*} 
for any $j \geq 1$. It would be interesting to find out any interpretation of this identity by means of the index of $(P,Q)$. 


\begin{thebibliography}{99}
\bibitem{AHT}\emph{W. N. Anderson, Jr, E. J. Harner, G. E. Trapp.} eigenvalues of the difference and product of projections. {\it Linear and Multilinear Algebra}. {\bf 17}, no. 3-4,  (1985), 295-299. 
\bibitem{Aub}\emph{Guillaume Aubrun}. Principal angles between random subspaces and polynomials in two free projections. {\it Confluentes Mathematici}. {\bf 13}, (2021), no. 2, 3-10. 
\bibitem{ASS0}\emph{J. Avron, R. Seiler, B. Simon}. Quantum Hall effect and the relative index for projections. {\it Phys. Rev. Lett}. {\bf 65}, (1990), no. 17, 220-237. 
\bibitem{ASS}\emph{J. Avron, R. Seiler, B. Simon}. The index of a pair of projections. {\it J. Func. Anal}. {\bf 120}, (1994), 220-237. 
\bibitem{Bak-Tre0}\emph{O. M. Baksalary, G. Trenkler}. Revisitation of the product of two orthogonal projectors. {\it Linear Algebra and its Applications}.  {\bf 431}, (2009), 2813-2833.
\bibitem{Bak-Tre1}\emph{O. M. Baksalary, G. Trenkler}. Eigenvalues of functions of orthogonal projectors. {\it Linear Algebra and its Applications}.  {\bf 431}, (2009), 2172-2186.
\bibitem{BL}\emph{F. Benaych-Georges, T. L\'evy}, Continuous semigroup of notions of independence between the classical and the free one. {\it Ann. Probab}. {\bf 39} (2011), no. 3, 904-938. 
\bibitem{Bia}\emph{P. Biane}. Free Brownian Motion, Free Stochastic Calculus and Random Matrices. {\it Fields Institute Communications, 12, (American Mathematical Society Providence, RI, 1997)}, pp. 1-19.
\bibitem{Che-Lan}\emph{Z. Che, B. Landon}. Local spectral statistics of the addition of random matrices. {\it Probab. Theory Rel. Fields}.  {\bf 175} (2019), no. 1-2, 579-654.
\bibitem{Dav}\emph{C. Davis}. Separation of two linear subspaces. {\it Acta Sci. Math}. {\bf 19} (1958), 172-187.
\bibitem{Dem}\emph{N. Demni}. Free Jacobi processes. {\it J. Theor. Proba}. {\bf 21} (2008), 118-143.
\bibitem{Dem1}\emph{N. Demni}. $\beta$-Jacobi processes. {\it Adv. Pure Appl. Math}. {\bf 1} (2010), no. 3, 325-344.
\bibitem{Dem-Hmi}\emph{N. Demni, T. Hmidi}. Spectral distribution of the free unitary Brownian motion: another approach. {\it S\'eminaire de Probabilit\'es XLIV, 191?206, Lecture Notes in Math}., 2046, Springer, Heidelberg, 2012. 
\bibitem{DHH}\emph{N. Demni, T. Hamdi, T. Hmidi}. Spectral distribution of the free Jacobi process. {\it Indiana Univ. Math. J}. {\bf 61} (2012), 1351-1368.
\bibitem{Gal}\emph{A. Galantai}. Subspaces, angles and pairs of orthogonal projections. {\it Linear and Multilinear Algebra}, {\bf Vol. 56}, No. 3, (2008), 227-260.
\bibitem{Hal}\emph{P. Halmos}. Two subspaces. {\it Trans. Amer. Math. Soc}. {\bf 144}, (1969), 381-389.
\bibitem{Ham1}\emph{T. Hamdi}. Liberation, free mutual information and orbital free entropy.  {\it Nagoya Math. J.}. (2018), 1-27.
\bibitem{Ham3}\emph{T. Hamdi}. Spectral distribution of the free Jacobi process, revisited. {\it Anal. PDE}. Vol. 11 {\bf 8} (2018), 2137-2148.
\bibitem{IU}\emph{M. Izumi, Y. Ueda}. Remarks on free mutual information and orbital free entropy. {\it Nagoya Math. J}. {\bf 220} (2015), 45-66.
\bibitem{Kar}\emph{V. Kargin}. On Eigenvalues of the Sum of Two Random Projections. {\it J. Stat. Phys.} (2012) {\bf 149}, 246-258.
\bibitem{Ner}\emph{Y. A. Neretin}. On Jordan angles and the triangle inequality in Grassmann manifolds. {\it Geom. Dedic.} {\bf 86}, (2001), 81-92.  
\bibitem{NS}\emph{A. Nica and R. Speicher}. Commutators of free random variables. {\it Duke Math. J.}, {\bf 92} (3): 553-592, 1998.
\bibitem{Nic-Spe}\emph{A. Nica, R. Speicher}. {\it Lectures on the combinatorics of free probability. London Mathematical Society Lecture Note Series}, Vol. 335, Cambridge University Press, 2006.
\bibitem{Oml}\emph{M. Omladic}. Spectral of the difference and product of projections. {\it Proc. A. M. S. } {\bf Volume 99}, Number 2, (1987), 317-318. 
\bibitem{Voi}\emph{D. V. Voiculescu}. The analogues of entropy and of Fisher's information measure in free probability theory. VI. Liberation and mutual free information. {\it Adv. Math.} {\bf 146} (1999), 101-166.
\bibitem{VDN}\emph{D. V. Voiculescu, K. J. Dykema, and A. Nica}. Free random variables. {\it CRM Monograph Series, volume 1}. American Mathematical Society, Providence, RI, 1992.
\end{thebibliography}
\end{document}